\documentclass[10pt]{smfart}

\RequirePackage[T1]{fontenc}
\RequirePackage{amsfonts,latexsym,amssymb}
\RequirePackage[frenchb]{babel}
\addto\extrasfrenchb{\bbl@nonfrenchitemize
\bbl@nonfrenchspacing}

\theoremstyle{plain}
\newtheorem{defi}{\definame}

\newtheorem{theo}[defi]{\theoname}

\newtheorem{lemm}[defi]{\lemmname}
\theoremstyle{definition}
\theoremstyle{remark}
\newtheorem{rema}[defi]{\remaname}

\usepackage[matrix,arrow]{xy}
\usepackage{url}

\let\cal\mathcal
\let\bb\mathbb

\let\epsilon\varepsilon

\def\Q{{\bf Q}} \def\Z{{\bf Z}}
\def\C{{\bf C}}
\def\N{{\bf N}}

\def\O{{\cal O}}

\def\wO{\widehat\O}

\def\B{\O{\bb B}_{\rm dR}^{\scriptscriptstyle +}} 
\def\Bdr{{\bb B}_{\rm dR}^{\scriptscriptstyle +}} 
\def\S{{\overline S}} 
\def\A{\O{\bb A}_{\rm inf}}
\def\Ainf{{\bb A}_{\rm inf}}

\begin{document}
\title{Une caract\'erisation de $\Bdr$ et $\B$}
\author{Pierre Colmez}
\address{CNRS, IMJ-PRG, Sorbonne Universit\'e, 4 place Jussieu, 75005 Paris, France}
\email{pierre.colmez@imj-prg.fr}
\begin{abstract}
Nous montrons
que ${\bf B}_{\rm dR}^+$ est l'\'epaississement universel de $\C_p$.
Plus g\'en\'eralement, nous montrons que, si $S$ est une alg\`ebre affino\"{i}de
r\'eduite, $\B(\S)$ est le $S$-\'epaississement
universel du compl\'et\'e de $\S$.
\end{abstract}
\begin{altabstract}
We show that
${\bf B}_{\rm dR}^+$ is the universal thickening of $\C_p$.
More generally, we show that, if $S$ is a reduced affinoid algebra,
$\B(\S)$ is the universal $S$-thickening
of the completion of $\S$.
\end{altabstract}

\maketitle
\let\languagename\relax

Fontaine a introduit l'anneau ${\bf B}_{\rm dR}^+$ dans~\cite{Fo1},
et syst\'ematis\'e la construction dans le cadre relatif dans~\cite{Fo2}. Il a
aussi explor\'e les propri\'et\'es universelles de ${\bf B}_{\rm dR}^+$ d'un point de vue galoisien
(cf.~\cite[prop.\,3.2]{Fo3}), ou celles de l'anneau ${\bf A}_{\rm inf}$, pierre angulaire 
de la construction de ${\bf B}_{\rm dR}^+$, du point de vue \'epaississement
infinit\'esimal~\cite{Fo2}.  Le th.\,\ref{1} ci-dessous\footnote{Qui
r\'epond \`a une question que m'a pos\'ee Alexander Petrov lors d'une conf\'erence \`a Carg\`ese.}
(ainsi que le th.\,\ref{4}) caract\'erise ${\bf B}_{\rm dR}^+$ de ce point de vue infinit\'esimal.

\section{Diff\'erentielles de K\"ahler et \'epaississement universel}
Soit $S$ une alg\`ebre affino\"{i}de r\'eduite\footnote{Le cas absolu
correspond \`a $S=\Q_p$, auquel cas $\wO_\S[\frac{1}{p}]=\C_p$.}, 
munie de la valuation spectrale $v_{\rm sp}$,
et soit $\O_S$
l'anneau de ses entiers (i.e.~l'ensemble des $f\in S$, v\'erifiant $v_{\rm sp}(f)\geq 0$).
Soit $\overline S$ l'extension \'etale maximale de $S$, et soient $\O_{\overline S}$
l'anneau de ses entiers et $\wO_\S$ son compl\'et\'e pour la topologie
$p$-adique.  Fontaine~\cite{Fo2} a d\'efini l'anneau $\A(\S)$ comme
le $\O_S$-\'epaississement universel de $\wO_\S$: 
on dispose d'un morphisme surjectif de $\O_S$-alg\`ebres $\theta:\A(\S)\to\wO_\S$ et, si
$\theta_A:A\to \wO_\S$ est {\it un $\O_S$-\'epaississement de $\wO_\S$} 
(i.e.~$A$ est une $\O_S$-alg\`ebre s\'epar\'ee et compl\`ete pour la topologie $p$-adique,
$\theta_A$ est surjectif, et il existe $k\in\N$ tel que
$({\rm Ker}\,\theta_A)^{k+1}=0$),
il existe
un unique morphisme de $\O_S$-alg\`ebres $\alpha_A:\A(\S)\to A$ tel que 
$\theta_A\circ\alpha_A=\theta$.

A partir de $\A(\S)$, on construit l'anneau $\B(\S)$ comme le compl\'et\'e
de $\A(\S)[\frac{1}{p}]$ pour la topologie ${\rm Ker}\,\theta$-adique (o\`u l'on a \'etendu
$\theta$ en un morphisme $\A(\S)[\frac{1}{p}]\to \wO_\S[\frac{1}{p}]$), et on \'etend
$\theta$ par continuit\'e en $\theta:\B(\S)\to \wO_\S[\frac{1}{p}]$.
Nous allons prouver (th.\,\ref{1}) que $\B(\S)$ est le $S$-\'epaississement universel
de $\wO_\S[\frac{1}{p}]$.

Avant d'\'enoncer le r\'esultat, rappelons que l'on dispose~\cite[th.\,4.3]{bald} 
\footnote{Il vaut mieux utiliser~\cite{bald}
que l'appendice de~\cite{Fo2} car celui-ci comporte un trou, comme il est expliqu\'e
dans~\cite[note\,1]{bald}.} 
d'une autre construction
de $\B(\S)$, comme compl\'et\'e
de $\S$ pour la topologie ci-dessous.
On d\'efinit, par r\'ecurrence, des sous-$\O_S$-alg\`ebres $\O^{(k)}$ de $\O_{\overline S}$, par
$$\O^{(0)}:=\O_{\overline S}, \quad \O^{(k+1)}:= 
{\rm Ker}\big(\O^{(k)}\to\O^{(0)}\otimes_{\O^{(k)}}\Omega_{\O^{(k)}/\O_S}\big)$$
Comme $\overline S$ est \'etale sur $S$, on a $\O^{(k)}[\frac{1}{p}]=\overline S$ pour tout $k$.
Soit alors $B^{(k)}:=\wO^{(k)}[\frac{1}{p}]$, o\`u
$\wO^{(k)}:=(\varprojlim_n \O^{(k)}/p^n)$, et soit $\B(\S):=
\varprojlim_k B^{(k)}$.  
Chaque $B^{(k)}$ est naturellement une $S$-alg\`ebre de Banach, et on
on munit $\B(\overline S)$ de la topologie de la limite projective, ce qui en fait
un fr\'echet. Notons que,
 par d\'efinition,
$$B^{(0)}=\wO_\S[\tfrac{1}{p}]$$

On appelle {\it $S$-\'epaississement de $\wO_\S[\frac{1}{p}]$} un morphisme
surjectif $\theta_B:B\to B^{(0)}$ de $S$-alg\`ebres de Banach tel
qu'il existe $k\in\N$ tel que
$({\rm Ker}\,\theta_B)^{k+1}=0$ (le plus petit tel $k$ est {\it l'ordre de
l'\'epaississement}).
\begin{theo}\label{1}
Si $\theta_B:B\to B^{(0)}$ est un $S$-\'epaississement,
il existe un unique morphisme continu $\alpha_B:\B(\S)\to B$ tel
que $\theta_B\circ\alpha_B=\theta$.
\end{theo}
\begin{proof}
Le lemme de Hensel permet de relever $\overline S$, de mani\`ere unique, dans~$B$.
L'unicit\'e de $\B(\overline S)\to B$ r\'esulte donc de la densit\'e de $\S$ dans 
$\B(\S)$.

Pour prouver le th\'eor\`eme, il suffit donc de prouver que, si $({\rm Ker}\,\theta_B)^{k+1}=0$,
on peut construire $B^{(k)}\to B$.  
Pour cela, il suffit de prouver que $\O^{(k)}\subset\S\subset B$ est born\'e dans $B$.
On raisonne par r\'ecurrence sur $k$.  

Soit $B'$ le quotient de $B$ par l'adh\'erence
$J$ de $({\rm Ker}\,\theta_B)^{k}$, de telle sorte que 
$B'\to B^{(0)}$ est un \'epaississement d'ordre $k-1$.
L'hypoth\`ese de r\'ecurrence implique que $\O^{(k-1)}$ est contenu dans un born\'e
$W'=p^{-N} B'_0$ de $B'$, o\`u $B'_0$ est
la boule unit\'e de $B'$. 
Soit $B_0$ la boule unit\'e de $B$; alors l'image de $B_0\to B'_0$ contient $p^{L}B'_0$ 
(th\'eor\`eme de 
l'image ouverte).
 Soit $W=p^{-N-L}B_0$; alors l'image de $W\to B'$ contient $W'$ 
et $J\cap W$ est un r\'eseau
de $J$.  Par continuit\'e de la multiplication, il existe $M$ tel que
$W\cdot W\subset p^{-M}W$.  
En particulier, l'action de $B$ sur $J$ se factorise \`a travers $B^{(0)}$ et
$\wO^{(0)}\cdot (W\cap J)\subset p^{-M}(W\cap J)$; on peut donc, quitte \`a remplacer
$W$ par $W+(\wO^{(0)}\cdot (W\cap J))$, supposer que $W\cap J$ est un $\wO^{(0)}$-module.

Si $x\in \O^{(k-1)}$, choisissons $\tilde x\in W$ ayant
pour image $x$ dans $B'$, et notons $\delta(x)$ l'image de $x-\tilde x$ dans 
le $\O^{(0)}$-module $J/p^{-M}(W\cap J)$; 
cette image ne d\'epend pas du choix de $\tilde x$ (c'est d\'ej\`a vrai
dans $J/(W\cap J)$).  De plus,
$\delta(xy)$ est l'image modulo
$p^{-M}(W\cap J)$ de $x(y-\tilde y)+(x-\tilde x)\tilde y-\widetilde{xy}+\tilde x\tilde y$. 
Or $\widetilde{xy}-\tilde x\tilde y\in p^{-M}W\cap J$, et donc
$\delta(xy)=x\delta(y)+y\delta(x)$.

Par la propri\'et\'e universelle de $\Omega_{\O^{(k)}/\O_S}$
et par d\'efinition de $\O^{(k)}$, la d\'erivation $\delta$
s'annule sur $\O^{(k)}$; autrement dit, on a $\O^{(k)}\subset W+p^{-M}(W\cap J)\subset p^{-M}W$,
ce qui permet de conclure.
\end{proof}

\section{De $\A$ \`a $\B$}
On donne (rem.\,\ref{3}) une seconde preuve du th.\,\ref{1} qui permet aussi de prouver
(th.\,\ref{4}) que $\Bdr(\S)$ est l'\'epaississement universel de $\wO_\S[\frac{1}{p}]$. 
\begin{rema}\label{3}
On peut aussi d\'eduire le th.\,\ref{1} de l'universalit\'e
de $\A(\S)$ en utilisant le lemme~\ref{2} ci-dessous\footnote{
Qui permet aussi d'all\'eger les hypoth\`eses de~\cite[prop.\,3.2]{Fo3}.}.
En effet, si $\theta_B:B\to B^{(0)}=\wO_\S[\frac{1}{p}]$ est un \'epaississement,
le lemme fournit 
une $\O_S$ alg\`ebre $A\subset B$, ouverte (et donc aussi ferm\'ee) et born\'ee (et donc
s\'epar\'ee et compl\`ete pour la topologie $p$-adique),
telle que $\theta_B$ induise une surjection $A\to \wO_\S$. L'universalit\'e de
$\A(\S)$ fournit un morphisme $\alpha:\A(\S)\to A$ tel que $\theta_B\circ\alpha_A$ soit
$\theta:\A(\S)\to\wO_\S$.
Si $({\rm Ker}\,\theta_B)^{k+1}=0$,
$\alpha$ se factorise par $\A(\S)/({\rm Ker}\,\theta)^{k+1}$, et donc induit
un morphisme 
$$\B(\S)/({\rm Ker}\,\theta)^{k+1}=
\A(\S)[\tfrac{1}{p}]/({\rm Ker}\,\theta)^{k+1}\to
\wO_\S[\tfrac{1}{p}]=B^{(0)}$$
  Ceci nous fournit le morphisme surjectif $\B(\S)\to B^{(0)}$ cherch\'e.

 l'unicit\'e d'un tel morphisme
est une cons\'equence de la densit\'e de $\S$ comme nous l'avons d\'ej\`a remarqu\'e.
Elle peut aussi se d\'eduire du (ii) du lemme~\ref{2}. En effet, si
$\alpha_1,\alpha_2$ sont deux tels morphismes, alors
$A_i:=\alpha_i(\A(\S))$ est une sous-$\O_S$-alg\`ebre born\'ee de $B$ pour $i=1,2$.
Il r\'esulte des (i) et (ii) du lemme~\ref{2} que l'on peut trouver un 
$\O_S$-\'epaississement
$A\subset B$ de $\wO_\S$ contenant $A_1$ et $A_2$.  L'universalit\'e
de $\A(\S)$ implique qu'il existe un unique $\alpha:\A(\S)\to A$ tel que
$\theta_B\circ\alpha=\theta$. Il en r\'esulte que $\alpha_1=\alpha_2$
sur $\A(\S)$, donc aussi (par $\Z_p$-lin\'earit\'e)
sur $\A(\S)[\frac{1}{p}]$, et donc aussi (par continuit\'e
et densit\'e de $\A(\S)[\frac{1}{p}]$) sur $\B(\S)$.
\end{rema}

\begin{lemm}\label{2}
{\rm (i)} Si $\theta_B:B\to\wO_\S[\frac{1}{p}]$ est un $S$-\'epaississement, alors
$B$ contient un sous-anneau ouvert et born\'e, qui est
un $\O_S$-\'epaississement de $\wO_\S$.

{\rm (ii)} Si $A_1\subset B$ est un $\O_S$-\'epaississement ouvert de
$\wO_\S$, et si $A_2$ est un sous-anneau born\'e de $B$,
il existe un $\O_S$-\'epaississement $A\subset B$ de
$\wO_\S$ contenant $A_1$ et $A_2$.
\end{lemm}
\begin{proof}
On fait une r\'ecurrence sur l'ordre $k$ de l'\'epaississement.
Si $k=0$, il n'y a rien \`a prouver.  Si $k\geq 1$, et si
$I={\rm Ker}(B\to\wO_\S[\frac{1}{p}])$ (on a $I^{k+1}=0$), soit $J$ l'adh\'erence
de $I^{k}$, et soit $B'=B/J$.  L'hypoth\`ese de r\'ecurrence fournit 
une $\O_S$-alg\`ebre $A'\subset B'$, ouverte (et donc aussi ferm\'ee)
et born\'ee, se surjectant sur $\wO_\S$.

Choisissons $W\subset B$, un $\O_S$-module ouvert born\'e, 
se surjectant sur $A'$ (on part d'un sous-$\O_S$-module born\'e de $B$ dont l'image
dans $B'$ contient $A'$, et on prend l'image inverse de $A'$ dans ce born\'e; l'existence
d'un tel born\'e r\'esulte de ce que la suite exacte $0\to J\to B\to B'\to 0$ est une suite
exacte stricte de banachs, cf.~preuve du th.\,\ref{1}).  Par continuit\'e de la multiplication,
il existe $M\in\N$ tel que $W\cdot W\subset p^{-M}W$; en particulier, $A'\cdot (W\cap J)\subset
p^{-M}(W\cap J)$ et on peut remplacer $W$ par $W+(A'\cdot (W\cap J))$ pour assurer
que $W\cap J$ soit un $A'$-module.
Soit alors $A=W+p^{-M}(W\cap J)$.  Si $a_1,a_2\in W$ et $b_1, b_2\in (W\cap J)$,
on a $(a_1+p^{-M}b_1)(a_2+p^{-M}b_2)=a_1a_2+p^{-M}(a_1b_2+a_2b_1)\in A$ car
$a_1a_2\in p^{-M}W\cap(W+J)=W+p^{-M}(W\cap J)=A$, et $a_1b_2+a_2b_1\in (W\cap J)$.
Il s'ensuit que $A$ est un sous-anneau ouvert born\'e de $B$, se surjectant sur $A'$ et donc
aussi sur $\wO_\S$.

Ceci prouve le (i). Maintenant, si $A_1\subset B$ est un $\O_S$-\'epaississement ouvert
de $\wO_\S$, et si $A_2$ est un sous-anneau born\'e de $B$,
il existe $N$ tel que $A_2\subset p^{-N}A_1$.  Mais $\theta_B(A_2)\subset \wO_\S$,
et donc $A_2\subset A_1+p^{-N}I_1$, o\`u $I_1=A_1\cap{\rm Ker}\,\theta_B$.
Soit $A=A_1+p^{-N}I_1+p^{-2N}I_1^2+\cdots$ (il n'y a qu'un nombre fini de termes
non nuls car $I_1^{k+1}=0$ si notre \'epaississement est d'ordre $k$). Alors
$A$ est un sous-anneau born\'e de $B$ qui contient $A_1$ et $A_2$.

Ceci permet de conclure.
\end{proof}

\begin{rema}
Au lieu de $\O_S$-\'epaississements de $\wO_\S$, on peut regarder les
\'epaississements $\theta_A:A\to\wO_\S$ de $\wO_\S$ (i.e.~on ne demande pas que 
$\theta_A$ soit un morphisme de $\O_S$-alg\`ebres ni m\^eme que $A$ soit une $\O_S$-alg\`ebre).
L'\'epaississement universel dans ce cadre est alors $\theta:\Ainf(\S)\to\wO_\S$ 
(cf.~\cite{Fo2}), et on d\'efinit
$\Bdr(\S)$ comme le compl\'et\'e de $\Ainf(\S)[\frac{1}{p}]$ pour la topologie ${\rm Ker}\,\theta$-adique.
Le lemme~\ref{2} et preuve du th.\,\ref{1} donn\'ee dans la rem.\,\ref{3} s'\'etendent
verbatim \`a ce cadre (contrairement \`a la premi\`ere preuve utilisant les diff\'erentielles
de K\"ahler).  On a donc le r\'esultat suivant.
\end{rema}

\begin{theo}\label{4}
$\Bdr(\S)$ est l'\'epaississement universel de $\wO_\S[\frac{1}{p}]$.
\end{theo}

\begin{rema}
Si on sp\'ecialise la situation au cas o\`u
$S$ est une extension finie $K$ de $\Q_p$, les th.\,\ref{1} et~\ref{4} 
deviennent:
${\bf B}_{\rm dR}^+$ est \`a la fois l'\'epaississement universel et le $K$-\'epaississement
universel de $\C_p$.
\end{rema}

\end{document}